\documentclass[12pt]{amsart}
\usepackage{latexsym,amssymb,amscd}

\def\B'c{{\mathcal{B'}}}
\def\U'c{{\mathcal{U'}}}

\def\opn#1#2{\def#1{\operatorname{#2}}} 
\opn\chara{char}
\opn\length{\ell}
\opn\cd{cd}
\opn\projdim{proj\,dim}
\opn\injdim{inj\,dim}
\opn\ini{in}
\opn\rank{rank}
\opn\depth{depth}
\opn\height{ht}
\opn\bigheight{bight}
\opn\embdim{emb\,dim}
\opn\codim{codim}

\opn\Tr{Tr}
\opn\bigrank{big\,rank}
\opn\superheight{superheight}\opn\lcm{lcm}
\opn\trdeg{tr\,deg}%
\opn\reg{reg}
\opn\lreg{lreg}
\opn\set{set}
\opn\supp{Supp}
\opn\shad{Shad}
\opn\indeg{indeg}
\opn\lex{lex}
\opn\div{div}
\opn\Div{Div}
\opn\cl{cl}
\opn\Cl{Cl}
\opn\Spec{Spec}
\opn\Supp{Supp}
\opn\supp{supp}
\opn\Sing{Sing}
\opn\Ass{Ass}
\opn\Mon{Mon}
\opn\Min{Min}
\opn\Ann{Ann}
\opn\Rad{Rad}
\opn\Soc{Soc}
\opn\Ker{Ker}
\opn\Coker{Coker}
\opn\Im{Im}
\opn\Hom{Hom}
\opn\Tor{Tor}
\opn\Ext{Ext}
\opn\End{End}
\opn\Aut{Aut}
\opn\id{id}

\opn\nat{nat}
\opn\GL{GL}
\opn\SL{SL}
\opn\mod{mod}
\opn\ord{ord}
\opn\ara{ara}
\opn\aff{aff}
\opn\con{conv}
\opn\relint{relint}
\opn\st{st}
\opn\lk{lk}
\opn\cn{cn}
\opn\core{core}
\opn\vol{vol}
\opn\gr{gr}

\def\pot#1#2{#1[\kern-0.28ex[#2]\kern-0.28ex]}

\opn\dirlim{\underrightarrow{\lim}}
\opn\invlim{\underleftarrow{\lim}}
\def\pnt{{\raise0.5mm\hbox{\large\bf.}}}

\def\Implies{\ifmmode\Longrightarrow \else
     \unskip${}\Longrightarrow{}$\ignorespaces\fi}
\def\implies{\ifmmode\Rightarrow \else
     \unskip${}\Rightarrow{}$\ignorespaces\fi}
\def\iff{\ifmmode\Longleftrightarrow \else
     \unskip${}\Longleftrightarrow{}$\ignorespaces\fi}

\let\:=\colon
\newtheorem{Theorem}{Theorem}[section]

\newtheorem{Corollary}[Theorem]{Corollary}
\newtheorem{Proposition}[Theorem]{Proposition}
\newtheorem{Remark}[Theorem]{Remark}

\newtheorem{Definition}[Theorem]{Definition}

\let\epsilon=\varepsilon
\let\phi=\varphi
\let\kappa=\varkappa
\numberwithin{equation}{section}

\title{Powers of edge ideals}

\author{Carmela Ferr\`o\and Mariella Murgia\and Oana Olteanu}

\address{Department of Mathematics, University of Messina, Via Ferdinando Stagno \newline D'Alcontres, Salita Sperone 31, 98166 Messina, Italy,} \email{cferro@unime.it}

\address{Department of Mathematics, University of Messina, Via Ferdinando Stagno \newline D'Alcontres, Salita Sperone 31, 98166 Messina, Italy,} \email{mariella.murgia@unime.it}

\address{Faculty of Mathematics and Computer Science, Ovidius University, Bd.\ Mamaia 124,
 900527 Constanta, Romania,} \email{olteanuoanastefania@gmail.com} 

\begin{document}

\maketitle

\begin{abstract}
We compute the Betti numbers for all the powers of initial and final lexsegment edge ideals. For the powers of the edge ideal of an anti--$d-$path, we prove that they have linear quotients and we characterize the normally torsion--free ideals. We determine a class of non--squarefree ideals, arising from some particular graphs, which are normally torsion--free.

Keywords: Betti number, associated prime ideal, edge ideal, normally torsion--free\\

MSC: 05C38, 13C99, 13A02
\end{abstract}

\section*{Introduction}
Graph theory have been intensively studied in the last years. It provides many interesting problems, being at the intersection of different areas of mathematics, such as commutative algebra, combinatorics, topology.

Let $G=(V,E(G))$ be a finite simple graph on the vertex set $V=\{1,\ldots, n\}$. To this combinatorial object, one may attach a squarefree monomial ideal, which is called the \textit{edge ideal}, whose minimal monomial generators are $x_ix_j$ with $\{i,j\}\in E(G)$. This allows us to describe combinatorial properties of the graph using an algebraic language. The edge ideal of a graph was first considered by R. Villarreal in \cite{V}.

An important class of graphs is given by the chordal ones. Chordal graphs have several characterizations, the most common being the following: a graph is \textit{chordal} if every cycle of length at least $4$ has a chord. By a chord of a cycle we mean an edge between two non--adjacent vertices of the cycle. One of the most important results is due to R. Fr\"oberg \cite{F}, who characterized all the edge ideals with a linear resolution in terms of the property of the complementary graph of being chordal. It naturally arises the same problem for all the powers of edge ideals which have a linear resolution. This characterization is due to J. Herzog, T. Hibi and X. Zheng \cite{HHZ}, who proved that the edge ideal has a linear resolution if and only if all its powers have a linear resolution. Moreover, this is equivalent with the edge ideal to have linear quotients. A more difficult problem is to find classes of graphs such that all the powers of the edge ideal have linear quotients. Some results in this sense were given by A.H. Hoefel and G. Whieldon \cite{HW}, E. Nevo and I. Peeva \cite{NP}.

A method to get useful information about the ideal is by determining the set of associated primes. It is known that for squarefree ideals, the set of associated primes coincides with the set of minimal primes. Moreover, the minimal primes of an edge ideal are precisely determined by the minimal vertex covers of the graph. When considering powers of an edge ideal $I\subset S=k[x_1,\ldots,x_n]$, it is known that $\Min(I)\subset \Ass_S(S/I^t)$, for all $t$. Moreover, it was proved \cite{BMV} that the set of associated primes of the powers of edge ideals form an ascending chain. A classical result in the commutative algebra, given by M. Broadmann \cite{B}, states that the set $\Ass_S(S/I^t)$ stabilizes for large $t$. If it became stabilized when $t=1$, then the ideal $I$ is called \textit{normally torsion--free}. There are two main problems concerning the set of associated prime ideals of $I^t$. The first one is to determine the prime ideals which belong to $\Ass_S(S/I^t)$, for all $t$. The second problem is to compute the index of stability, meaning to determine the minimal integer $t$ such that $\Ass_S(S/I^t)$ stabilizes.

In this paper, we analyze, for some classes of graphs, these two kind of problems. Firstly, we describe the relation between the Betti numbers of the edge ideal and the Betti numbers of its powers. This is done by applying the formula for computing the Betti numbers of an ideal with linear quotients. Secondly, we determine a class of non--squarefree ideals, arising from some particular graphs, which are normally torsion--free.

The paper is structured as follows. The second section contains the basic definitions and some useful results.

In Section $3$, we compute the Betti numbers for the cases when the edge ideal is an initial and a final squarefree lexsegment ideal generated in degree $2$. We prove that all the powers of initial and final squarefree lexsegment ideals generated in degree $2$ have linear quotients, Proposition \ref{set-initial} and Proposition \ref{set-final}. As an application, we compute the Betti numbers of powers of such ideals.

In Section $4$ we pay attention to a particular class of chordal graphs, namely to $d-$path graphs. The complementary graph of a $d-$path is called an \textit{anti--$d-$path} and its edge ideal has a linear resolution. We prove that all the powers of the edge ideal of an anti--$d-$path have linear quotients. Moreover, we describe the set of associated primes of the powers of the edge ideal of an anti--$d-$path, and we characterize those which are normally torsion--free. 
	\[
\]
\textbf{Acknowledgment.} The authors would like to thank Professors Ralf Fr\"oberg and Mats Boij for valuable discussions and comments during the preparation of this paper. The authors are grateful to the organizers of the the School of Research PRAGMATIC 2011, Catania, Italy. 

\section{Background}

Let $S=K[x_1,\ldots,x_n]$ be the polynomial ring in $n$ variables over a field $K$. We order the monomials in $S$ lexicographically with $x_1>_{lex}\cdots>_{lex}x_n$. For a monomial $u\in S$, we set $\max(u)=\max(\supp(u))$ and $\min(u)=\min(\supp(u))$, where $\supp(u)=\{i:x_i\mid u\}$. Moreover, we will denote by $\nu_s(u)$ the exponent of the variable $x_s$ in the monomial $u$.

For a monomial ideal $I\subset S$, we will denote by $G(I)$ the set of minimal monomial generators of $I$. 

A monomial ideal $I$ of $S$ has \textit{linear quotients} if the monomials from the minimal monomial set of generators of $I$ can be ordered $u_1,\ldots, u_s$ such that for all $2\leq i\leq s$ the colon ideals $(u_1,\ldots,u_{i-1}):u_i$ are generated by variables. In this case, we will denote by $\set(u_i)=\{x_j:x_j\in(u_1,\ldots,u_{i-1}):u_i\}$.

The Betti numbers of ideals with linear quotients are given in \cite{HH}:

\begin{Proposition}\cite{HH}\label{betti}
Let $I\subset S$ be a graded ideal with linear quotients generated in one degree. Then
	\[\beta_i(I)=\sum\limits_{u\in G(I)}{|\set(u)|\choose i}.
\]
\end{Proposition} 

It is known, \cite{CH} that any monomial ideal generated in one degree, which has linear quotients, has a linear resolution. In \cite{HHZ}, the monomial ideals generated in degree $2$ with a linear resolution are described.

\begin{Theorem}\cite{HHZ}
Let $I$ be a monomial ideal generated in degree $2$. The following conditions are equivalent:
\begin{itemize}
\item [(a)] $I$ has a linear resolution;
\item [(b)] $I$ has linear quotients;
\item [(c)] Each power of $I$ has a linear resolution.
\end{itemize}
\end{Theorem}

In the following, we will consider squarefree monomial ideals generated in degree $2$. In general, to a squarefree monomial ideal generated in degree $2$ one may associate a graph $G=(V,E(G))$ on the vertex set $V=[n]$ such that $I=I(G)$ is its edge ideal, that is the ideal generated by the squarefree monomials $x_ix_j$, with $\{i,j\}\in E(G)$. The edge ideals with a linear resolution are described in \cite{F}.

\begin{Proposition}\cite{F}
Let $G$ be a graph and $\bar G$ its complementary graph. Then $I(G)$ has a linear resolution if and only if $\bar G$ is chordal.
\end{Proposition}

For edge ideals $I=I(G)$, in \cite{BMV} it is proved that the sets of associated prime ideals of powers of $I$ form an ascending chain. In \cite{B} Brodmann proved that $\Ass_S(S/I^k)$ stabilizes for large $k$, that is there is an integer $N$ such that $\Ass_S(S/I^k)=\Ass_S(S/I^N)$, for all $k\geq N$. The ideal $I$ is called \textit{normally torsion--free} if $\Ass_S(S/I)=\Ass_S(S/I^k)$, for all $k\geq 1$. The normally torsion--free edge ideals are precisely those ideals associated to bipartite graphs, \cite{SVV}. We recall that a graph $G$ is \textit{bipartite} if its vertex set is the disjoint union of the sets $V_1$ and $V_2$, such that each edge of $G$ has one vertex in $V_1$ and the other one in $V_2$.  

Although the normally torsion--free squarefree ideals were studied in a series of papers, the non--squarefree case it is still unknown. In this sense, we will determine a class of non--squarefree ideals which are normally torsion--free.

\section{Initial and final lexsegment edge ideals}

Firstly, we are interested in computing the Betti numbers of the powers of an initial squarefree lexsegment ideal generated in degree $2$. We recall their definition.

\begin{Definition}\rm
Let $v=x_ix_j$ be a squarefree monomial in $S$. The \textit{initial lexsegment set} defined by $v$ is the set
$$L^i(v)=\{w: w \mbox{ is a squarefree monomial of degree } 2,\ w\geq_{lex}v\}.$$
An ideal generated by an initial squarefree lexsegment set is called an \textit{initial lexsegment edge ideal}.
\end{Definition}

\begin{Proposition}\label{set-initial}
Let $I=(L^i(v))$ be an initial lexsegment edge ideal. For $t\geq 1$, we denote by $G(I^t)=\{u_1,\ldots, u_m\}$, with $u_1>_{lex}\cdots>_{lex}u_m$. Then 
$$(u_1,\ldots, u_{i-1}):(u_i)=(x_r: \nu_r(x_ru_i)\leq t, \mbox{ for all }1\leq r\leq \max(u_i)-1),$$
for all $2\leq i\leq m$.
\end{Proposition}

\begin{proof}
Let $m\in(u_1,\ldots,u_{i-1}):(u_i)$ be a monomial. Then there is a minimal monomial generator $u_j>_{lex}u_i$ such that $u_j\mid mu_i$. We want to prove that there exists a variable $x_r$, with $1\leq r\leq \max(u_i)-1$, and $\nu_r(x_ru_i)\leq t$ with the property that $x_r\mid m$. Since $u_j>_{lex}u_i$, it results that there is an integer $l\geq 1$ such that $\nu_s(u_j)=\nu_s(u_i)$, for all $s<l$ and $\nu_l(u_j)>\nu_l(u_i)$. The condition $\nu_l(u_j)>\nu_l(u_i)$ yields to $x_l\mid m$, since $u_j\mid mu_i$. By the relation $\deg(u_j)=\deg(u_i)$, we obtain $l<\max(u_i)$. Moreover, $\nu_l(x_lu_i)=\nu_l(u_i)+1\leq \nu_l(u_j)\leq t$, since $u_j\in G(I^t)$. Therefore we proved that the variable $x_l$ satisfies the desired conditions.   

Conversely, let $1\leq r\leq \max(u_i)-1$, with $\nu_r(x_ru_i)\leq t$. We want to prove that $x_r\in (u_1,\ldots, u_{i-1}):(u_i)$. Consider the monomial $u_j=x_ru_i/x_{\max(u_i)}$. Then it is clear that $u_j>_{lex}u_i$ and $u_j\mid x_ru_i$. It remains to argue that $u_j\in G(I^t)$.   

Since $u_i\in G(I^t)$, we have $u_i=m_1\cdots m_t$, with $m_1\geq_{lex}\ldots\geq_{lex}m_t\geq_{lex}v$. By hypothesis, $\nu_r(x_ru_i)\leq t$, thus there is some integer $1\leq s\leq t$ such that $x_r\nmid m_s$. We study two cases:

\textit{Case 1.} If $x_{\max(u_i)}\mid m_s$, then
$$u_j=x_ru_i/x_{\max(u_i)}=m_1\cdots m_{s-1}\frac{x_rm_s}{x_{\max(u_i)}}m_{s+1}\cdots m_t\in G(I^t),$$
since $x_rm_s/x_{\max(u_i)}>_{lex}m_s\geq_{lex} v$.

\textit{Case 2.} Assume that $x_{\max(u_i)}\nmid m_s$, that is $m_s=x_{\alpha}x_{\beta}$, with $\alpha<\beta<\max(u_i)$ and $\alpha\neq r,\ \beta\neq r$. Consider the monomial $m_k=x_{\gamma}x_{\max(u_i)}$, with $\gamma<\max(u_i)$, for some $1\leq k\neq s\leq t$. It is clear that if $\gamma\neq r$, then $x_rx_{\gamma}\geq_{lex}x_{\gamma}x_{\max(u_i)}=m_k\geq_{lex}v$. Hence 
$$u_j=x_ru_i/x_{\max(u_i)}=m_1\cdots m_{k-1}\frac{x_rm_k}{x_{\max(u_i)}}m_{k+1}\cdots m_t\in G(I^t).$$
Otherwise, if $\gamma=r$, then $x_{\alpha}x_{r}\geq_{lex}x_{\gamma}x_{\max(u_i)}\geq_{lex}v$ and $x_{\beta}x_{\gamma}\geq_{lex}x_{\gamma}x_{\max(u_i)}=m_k\geq_{lex}v$. Then 
$$u_j=x_ru_i/x_{\max(u_i)}=\left(\prod\limits_{q\neq s,q\neq k}m_q\right)(x_{\alpha}x_{r})(x_{\beta}x_{\gamma})\in G(I^t),$$
which ends the proof.
\end{proof}

\begin{Corollary}
Let $I$ be an initial lexsegment edge ideal. Denote by $G(I^t)=\{u_1,\ldots, u_m\}$, with $u_1>_{lex}\cdots>_{lex}u_m$, for all $t\geq 1$. Then 
	\[|\set(u_i)|=\left\{\begin{array}{cc}
	\max(u_i)-2, & \mbox{ if } x_j^t\mid u_i, \mbox{ for some }j<\max(u_i)\\
	\max(u_i)-1, & \mbox{ otherwise}\\
	\end{array}\right.
\]
for all $1\leq i\leq m$.
\end{Corollary}

\begin{proof}
Let $u_i\in I^t$ be a minimal monomial generator. By Proposition \ref{set-initial}, one has 
$$\set(u_i)=\{x_r: \nu_r(x_ru_i)\leq t, \mbox{ for all }1\leq r\leq \max(u_i)-1\}.$$
If there is some integer $1\leq j<\max(u_i)$ such that $\nu_j(u_i)=t$, since $\deg(u_i)=2t$ and the exponents of all variables from the support of $u_i$ are at most $t$, we obtain
$$\set(u_i)=\{x_1,\ldots,x_{\max(u_i)-1}\}\setminus\{x_j\},$$
thus $|\set(u_i)|=\max(u_i)-2$.

Otherwise, we have $\nu_s(u_i)<t$, for all $s\in \supp(u_i)$, $s<\max(u_i)$, and we obtain
$$\set(u_i)=\{x_1,\ldots,x_{\max(u_i)-1}\},$$
thus $|\set(u_i)|=\max(u_i)-1$.
\end{proof}

\begin{Corollary}
Let $I$ be an initial lexsegment edge ideal. For all $t\geq 1$, denote by $G(I^t)=\{u_1,\ldots, u_m\}$, with $u_1>_{lex}\cdots>_{lex}u_m$. Then 
	\[\beta_i(I)=\sum\limits_{j=1}^{m}{\max(u_j)-2\choose i},\]
		\[\beta_i(I^t)=\sum\limits_{j=1}^{m}\left({\max(u_j)-1\choose i}+{\max(u_j)-2\choose i}\right),\mbox{ for }t>1.\]
\end{Corollary}

\begin{Remark}\rm
Let $G$ be the star graph on the vertex set $[n]$ with the edge ideal $I=(x_1x_2,x_1x_3,\ldots,x_1x_n)$. It is clear that $I$ is the initial lexsegment edge ideal determined by the monomial $v=x_1x_n$. For $t\geq 1$, we note that $I^t=x_1^t(x_2,x_3,\ldots,x_n)^t$. Moreover, any minimal monomial generator $u$ of $I^t$ is divisible by $x_1^t$, thus $|\set(u)|=\max(u)-2$. Therefore
\[\beta_i(I^t)=\sum\limits_{u\in G(I^t)}{\max(u)-2\choose i}.\]
It is easy to see that 
$$|\{u\in G(I^t): \max(u)=j\}|=|\{w\in \Mon(k[x_2,\ldots, x_j]):\deg(w)=t\}|=$$$$={j+t-2\choose t}.$$
Then \[\beta_i(I^t)=\sum\limits_{j=2}^{n}{j+t-2\choose t}{j-2\choose i}.\]
\end{Remark}

Next, we are interested in computing the Betti numbers of the powers of a final squarefree lexsegment ideal generated in degree $2$. 

\begin{Definition}\rm
Let $u=x_ix_j$ be a squarefree monomial in $S$. The \textit{final lexsegment set} defined by $u$ is the set
$$L^f(u)=\{w: w \mbox{ is a squarefree monomial of degree } 2,\ u\geq_{lex}w\}.$$
An ideal generated by a final squarefree lexsegment set is called a \textit{final lexsegment edge ideal}.
\end{Definition}

\begin{Proposition}\label{set-final}
Let $I=(L^f(u))$ be a final lexsegment edge ideal and $G(I^t)=\{u_1,\ldots, u_m\}$, with $u_1<_{revlex}\cdots<_{revlex}u_m$ be the set of minimal monomial generators of $I^t$, for $t\geq 1$. Then 
$$(u_1,\ldots, u_{i-1}):(u_i)=(x_r: \nu_r(x_ru_i)\leq t, \mbox{ for all } r\geq \min(u_i)+1),$$
for all $2\leq i\leq m$.
\end{Proposition}

\begin{proof}
Let $m\in(u_1,\ldots,u_{i-1}):(u_i)$ be a monomial. Then there is a minimal monomial generator $u_j<_{revlex}u_i$ such that $u_j\mid mu_i$. By $u_j<_{revlex}u_i$ we get that there is an integer $l\geq 1$ such that $\nu_s(u_j)=\nu_s(u_i)$, for all $s>l$ and $\nu_l(u_j)>\nu_l(u_i)$. Since $u_j\mid mu_i$ and $\nu_l(u_j)>\nu_l(u_i)$, it results that $x_l\mid m$. It is clear that $l>\min(u_i)$ by degree considerations. Moreover, $\nu_l(x_lu_i)=\nu_l(u_i)+1\leq \nu_l(u_j)\leq t$, since $u_j\in G(I^t)$. Therefore $x_l$ satisfies the desired conditions.   

Conversely, let $r\geq \min(u_i)+1$, with $\nu_r(x_ru_i)\leq t$. We want to prove that $x_r\in (u_1,\ldots, u_{i-1}):(u_i)$. We take the monomial $u_j=x_ru_i/x_{\min(u_i)}$. It is clear that $u_j<_{revlex}u_i$ and $u_j\mid x_ru_i$. It remains to argue that $u_j\in G(I^t)$.   

Since $u_i\in G(I^t)$, we have $u_i=m_1\cdots m_t$, with $m_1,\ldots, m_t\in L^f(u)$. We may assume that $m_1=x_{\min(u_i)}x_a$, with $a>\min(u_i)$. If $a\neq r$, then
$$u_j=x_ru_i/x_{\min(u_i)}=(x_rx_a)m_2\cdots m_t\in G(I^t),$$
since $u\geq_{lex}m_1=x_{\min(u_i)}m_1/x_{\min(u_i)}>_{lex}x_rm_1/x_{\min(u_i)}=x_rx_a$.

Assume that $a=r$. By hypothesis, $\nu_r(x_ru_i)\leq t$, thus there is some integer $1\leq s\leq t$ such that $x_r\nmid m_s$. We denote $m_s=x_px_q$ and we note that $p,q\neq r$. Then
$$u_j=x_ru_i/x_{\min(u_i)}=(x_rx_p)(x_rx_q)m_2\cdots m_{s-1}m_{s+1}\cdots m_t\in G(I^t),$$
since $u\geq_{lex}m_1=x_{\min(u_i)}x_r\geq_{lex}x_rx_p$ and $u\geq_{lex}m_1=x_{\min(u_i)}x_r\geq_{lex}x_rx_q$.
\end{proof}

\begin{Corollary}
Let $I$ be a final lexsegment edge ideal and $G(I^t)=\{u_1,\ldots, u_m\}$, with $u_1<_{revlex}\cdots<_{revlex}u_m$, for all $t\geq 1$. Then 
	\[|\set(u_i)|=\left\{\begin{array}{cc}
  n-\min(u_i)-1, & \mbox{ if } x_j^t\mid u_i, \mbox{ for some }j>\min(u_i)\\
	n-\min(u_i), & \mbox{ otherwise}\\
	\end{array}\right.
\]
for all $1\leq i\leq m$.
\end{Corollary}

\begin{proof}
Let $u_i\in I^t$ be a minimal monomial generator. Then 
$$\set(u_i)=\{x_r: \nu_r(x_ru_i)\leq t, \mbox{ for all }r\geq \min(u_i)+1\},$$
by Proposition \ref{set-final}. If there is some integer $j>\min(u_i)$ such that $\nu_j(u_i)=t$, since $\deg(u_i)=2t$ and the exponents of all variables from the support of $u_i$ are at most $t$, we obtain
$$\set(u_i)=\{x_{\min(u_i)+1},\ldots,x_{n}\}\setminus\{x_j\},$$
thus $|\set(u_i)|=n-\min(u_i)-1$.

Otherwise, we have $\nu_s(u_i)<t$, for all $s\in \supp(u_i)$, $s>\min(u_i)$, and we obtain
$$\set(u_i)=\{x_{\min(u_i)+1},\ldots,x_{n}\},$$
thus $|\set(u_i)|=n-\min(u_i)$.
\end{proof}

\begin{Corollary}
Let $I$ be a final lexsegment edge ideal and $G(I^t)=\{u_1,\ldots, u_m\}$, with $u_1<_{revlex}\cdots<_{revlex}u_m$. Then 
	\[\beta_i(I)=\sum\limits_{j=1}^{m}{n-\min(u_j)-1\choose i},\]
		\[\beta_i(I^t)=\sum\limits_{j=1}^{m}\left({n-\min(u_j)\choose i}+{n-\min(u_j)-1\choose i}\right),\mbox{ for }t>1.\]
\end{Corollary}

\section{The edge ideal of anti--$d-$path}

In this section, we will study properties of the edge ideal of the complement of a $d-$path with the set of vertices $[n]$.

We will follow the definition of a $d-$path given in \cite{Fe}.

\begin{Definition}\rm
Let $d\geq 1$ be an integer. A \textit{$d-$path} is a graph on the vertex set $\{1,\ldots,n\}$ which is the union
of the complete graphs on the vertex sets $\{1,\ldots,d+1\}$, $\{2,\ldots,d+2\},\ldots,\ \{n-d,\ldots,n\}$.
\end{Definition}
It is clear by definition that a $1-$path is a simple path, while a $2-$path is a graph of the form:
	\[
\]
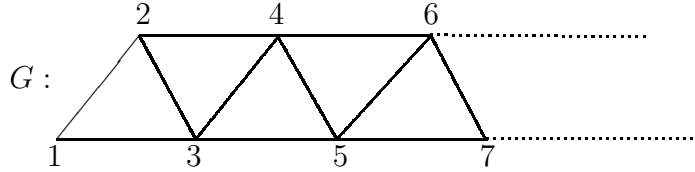
\begin{center}

\unitlength 1mm 
\linethickness{0.8pt}
\ifx\plotpoint\undefined\newsavebox{\plotpoint}\fi 
\begin{picture}(93.25,24)(0,0)
\put(7,7.25){\line(4,5){11}}
\multiput(18,21)(.033632287,-.061659193){223}{\line(0,-1){.061659193}}
\multiput(25.5,7.25)(.0336826347,.0419161677){334}{\line(0,1){.0419161677}}
\multiput(36.5,20.75)(.033695652,-.058695652){230}{\line(0,-1){.058695652}}
\multiput(44.25,7.25)(.0336927224,.0370619946){371}{\line(0,1){.0370619946}}
\multiput(56.75,21)(.03372093,-.063953488){215}{\line(0,-1){.063953488}}
\multiput(56.68,21)(.97414,-.00862){30}{{\rule{.8pt}{.8pt}}}
\put(7,7.25){\line(1,0){57.25}}
\put(64.25,7.25){\line(-1,0){.25}}
\put(18,21){\line(1,0){38.75}}
\multiput(64.18,7.25)(.93333,0){30}{{\rule{.8pt}{.8pt}}}
\put(6.75,5){\makebox(0,0)[cc]{$1$}}
\put(25.25,5){\makebox(0,0)[cc]{$3$}}
\put(44.75,5){\makebox(0,0)[cc]{$5$}}
\put(64.25,5){\makebox(0,0)[cc]{$7$}}
\put(18.5,24){\makebox(0,0)[cc]{$2$}}
\put(36.25,24){\makebox(0,0)[cc]{$4$}}
\put(56.75,24){\makebox(0,0)[cc]{$6$}}
\put(3.5,15.5){\makebox(0,0)[cc]{$G:$}}
\end{picture}
\end{center}
	\[
\]
The $d-$paths are particular cases of $d-$trees. Moreover, in \cite{Fe} it is proved that the edge ideal of the complement of a $d-$tree is Cohen--Macaulay.

Let $G$ be a $d-$path on the vertex set $V(G)=\{1,\ldots,n\}$. The complementary graph of $G$, denoted by $\bar G$, is called \textit{anti--$d-$path}. The edge ideal of the complementary graph of $G$ is
$$I=I(\bar{G})=(x_ix_j:i+d<j,\ i,j\in V(G)).$$
Indeed, since the graph $G$ is the union of the complete graphs on the vertex sets $\{1,\ldots,d+1\}$, $\{2,\ldots,d+2\},\ldots,\ \{n-d,\ldots,n\}$, we obviously have $\{i,j\}\in E(G)$, for all $i,j\in V(G)$, with $i<j\leq i+d$.
 
In the following, we are interested in computing the Betti numbers for the powers of the ideal $I$. Firstly, we will describe the minimal monomial generating set for all the powers of the edge ideal $I(\bar G)$. The next two propositions represent the generalization of some results given in \cite{HW}.

\begin{Proposition}\label{min-mon}
For all $k\geq 1$,
$$G(I^k)=\{x_{i_1}\cdots x_{i_k}x_{j_1}\cdots x_{j_k}:\ i_1\leq\cdots\leq i_k\leq j_1\leq\cdots\leq j_k,\ i_r+d<j_r,1\leq r\leq k\}.$$
\end{Proposition}

\begin{proof}
For the inclusion "$\subseteq$", we consider $m\in G(I^k)$. Since $\deg(m)=2k$, we may write $m=x_{i_1}\cdots x_{i_k}x_{j_1}\cdots x_{j_k}$, with $i_1\leq\cdots\leq i_k\leq j_1\leq\cdots\leq j_k$. Assume by contradiction that there is an integer $1\leq r\leq k$ such that $i_r+d\geq j_r$. Since $i_r\leq\cdots\leq i_k\leq j_1\leq\cdots\leq j_r$ and $j_r\leq i_r+d$, we obtain that 
$$\{i_r,\ldots, i_k,j_1,\ldots, j_r\}\subseteq\{i_r,i_r+1,\ldots,i_r+d\}.$$
Let $w=x_{i_r}\cdots x_{i_k}x_{j_1}\cdots x_{j_r}$. Then $w\mid m$ and $\supp(w)\subseteq \{i_r,i_r+1,\ldots,i_r+d\}$. Hence $w\notin G(I^k)$ and $\deg(w)=k+1$.

By hypothesis, $m$ is a product of $k$ minimal monomial generators of $\bar{G}$, thus every divisor of degree $k+1$ of $m$ must contain at least one edge. But the construction of the monomial $w$ contradicts this statement, thus $i_r+d<j_r$, for all $1\leq r\leq k$.

The other inclusion is clear.
\end{proof}

\begin{Proposition}
For all integers $k\geq 1$, the ideal $I(\bar G)^k$ has linear quotients with respect to the decreasing lexicographical order of its minimal monomial generators.
\end{Proposition}

\begin{proof}
Let $m'>_{lex}m$ be two minimal monomial generators of $I^k=I(\bar G)^k$. By Proposition \ref{min-mon}, one has
$$m=x_{i_1}\cdots x_{i_k}x_{j_1}\cdots x_{j_k}$$
$$m'=x_{s_1}\cdots x_{s_k}x_{t_1}\cdots x_{jt_k}$$
with $i_1\leq\cdots\leq i_k\leq j_1\leq\cdots\leq j_k$, $s_1\leq\cdots\leq s_k\leq t_1\leq\cdots\leq t_k$ and $i_r+d<j_r$, $s_r+d<t_r$, for all $1\leq r\leq k$.

We want to prove that the monomial $m'/\gcd(m',m)$ is divisible by some variable $x_j=m''/\gcd(m'',m)$, for some $m''>_{lex}m$. We will analyze two cases:

\textit{Case 1:} If there is some $q\geq 1$ such that $i_l=s_l$, for all $l<q$ and $i_q>s_q$, then we consider the monomial
$$m''=x_{s_q}m/x_{i_q}=x_{i_1}\cdots x_{i_{q-1}}x_{s_q}x_{i_{q+1}}\cdots x_{i_k}x_{j_1}\cdots x_{j_k}.$$
It is clear that $m''>_{lex}m$, and $m''\in G(I^k)$ since $s_q+d<i_q+d<j_q$.

\textit{Case 2:} Assume that $i_r=s_r$, for all $1\leq r\leq k$ and there is some $q\geq 1$ such that $j_l=t_l$, for all $l<q$ and $j_q>t_q$. We construct the monomial
$$m''=x_{t_q}m/x_{j_q}=x_{i_1}\cdots x_{i_k}x_{j_1}\cdots x_{j_{q-1}}x_{t_q}x_{j_{q+1}}\cdots x_{j_k}.$$
It is clear that $m''>_{lex}m$, and $m''\in G(I^k)$ since $i_q+d=s_q+d<t_q$.
\end{proof}

\begin{Proposition}
Let $k\geq 1$ and $u=x_{i_1}\cdots x_{i_k}x_{j_1}\cdots x_{j_k}$ be a minimal monomial generator of $G(I^k)$. Then
$$\set(u)=\{x_1,\ldots,x_{i_k-1}\}\cup\bigcup\limits_{1\leq r\leq k}\{x_s:i_r+d<s<j_r\}.$$
\end{Proposition}

\begin{proof}
For the inclusion "$\subseteq$", let $m\in G(I^k)$, $m>_{lex}u$, $m=x_{a_1}\cdots x_{a_k}x_{b_1}\cdots x_{b_k}$. We will prove that there is an integer $1\leq t\leq i_k-1$, there is a monomial $m_1\in G(I^t)$, $m_1>_{lex}u$ such that $m_1/\gcd(u,m_1)=x_t$ and $x_{t}\mid m/\gcd(u,m)$ or there exist $1\leq r\leq k$, $i_r+d<s<j_r$ and $m_2\in G(I^t)$, $m_2>_{lex}u$ such that $m_2/\gcd(u,m_2)=x_s$ and $x_{s}\mid m/\gcd(u,m)$.

Since $m>_{lex}u$, we will analyze the following two cases: 

\textit{Case 1.} Assume that there is some $q\geq 1$ such that $i_l=a_l$, for all $l<q$ and $a_q<i_q$. Consider the monomial $m_1=x_{a_q}u/x_{i_q}$. One has $m_1>_{lex}u$ and $m_1/\gcd(u,m_1)=x_{a_q}$. Since $a_q+d<i_q+d<j_q$, we have $m_1\in G(I^k)$.

Moreover, one has $x_{a_q}\mid m/\gcd(u,m)$ and $a_q<i_q\leq i_k$. 

\textit{Case 2.} If $i_r=a_r$, for all $1\leq r\leq k$ and there is some $q\geq 1$ such that $b_l=j_l$, for all $l<q$ and $b_q<j_q$, then we take the monomial $m_2=x_{b_q}u/x_{j_q}$. It is clear that $m_2>_{lex}u$, $m_2/\gcd(u,m_2)=x_{b_q}$ and $m_2\in G(I^k)$, since $i_q+d=a_q+d<b_q<j_q$. Moreover, $x_{b_q}\mid m/\gcd(u,m)$.

For the inclusion "$\supseteq$", firstly, let $1\leq t\leq i_k-1$ and the monomial $m=x_tu/x_{i_k}$. Then $m>_{lex}u$ and $m\mid x_tu$. Moreover, we have $m\in G(I^k)$. Indeed, if $i_l\leq t\leq i_{l+1}$, for some $1\leq l<k$, then $i_l+d\leq t+d\leq i_{l+1}+d<j_{l+1}$ and $i_s+d<j_s$, for all $s\neq l$.

Secondly, let $1\leq r\leq k$, $i_r+d<s<j_r$ and consider the monomial $m=x_su/x_{j_k}$. One has that $m>_{lex}u$ and $m\mid x_su$. The monomial $m\in G(I^k)$, since for all $1\leq t\neq r\leq k$ we have $i_t+d<j_t$ and $i_r+d<s$.
\end{proof}

Using Proposition \ref{betti}, one may compute the Betti numbers of the edge ideal of an anti--$d-$path.

\medskip

Next, we describe the minimal vertex covers of an anti--$d-$path.

\begin{Proposition}
Let $\bar G$ be an anti--$d-$path and $I=I(\bar G)$ be its edge ideal. Then the minimal primary decomposition of $I$ is
$$I=\bigcap\limits_{t=1}^{n-d}P_{[n]\setminus\{t,t+1,\ldots,t+d\}},$$
where $P_{[n]\setminus\{t,t+1,\ldots,t+d\}}=(x_s: s\in [n]\setminus\{t,t+1,\ldots,t+d\})$.
\end{Proposition}

\begin{proof}
Since the minimal vertex covers of $\bar G$ corresponds to the maximal independent sets of $\bar G$, it is enough to show that all the maximal independent sets of $\bar G$ are $\{t,t+1,\ldots,t+d\}$, with $1\leq t\leq n-d$.

Let $1\leq t\leq n-d$ and $A=\{t,t+1,\ldots,t+d\}$. Then $A$ is a maximal independent set since $E(\bar G)=\{\{i,j\}:j-i>d\}$.

Let $B$ be a maximal independent set of $\bar G$. Then for all $i,j\in B$, we have $\{i,j\}\notin E(\bar G)$, that is $\{i,j\}\in E(G)$. But the graph $G$ is the union of the complete graphs on the vertex sets $\{1,\ldots,d+1\}$, $\{2,\ldots,d+2\},\ldots,\ \{n-d,\ldots,n\}$. Therefore, $B\subset\{t,\ldots,d+t\}$, for some $1\leq t\leq n-d$. Since $B$ is a maximal independent set, we must have $B=\{t,\ldots,d+t\}$.
\end{proof}

Using this, we may recover a result from \cite{Fe}.

\begin{Corollary}
The edge ideal of an anti--$d-$path is Cohen--Macaulay of dimension $d+1$. 
\end{Corollary}

\begin{proof}
By the minimal primary decomposition, it results that the edge ideal of an anti--$d-$path is of height $n-d-1$. Using \cite[Theorem 3.3]{Fe}, it follows the assertion.
\end{proof}

In the following, we characterize the edge ideals of anti--$d-$paths which are normally torsion--free.

\begin{Theorem}
Let $\bar G$ be an anti--$d-$path and $I=I(\bar G)$ be its edge ideal. Then for all $k>1$
	\[\Ass_S(S/I^k)=\left\{\begin{array}{ll}
	\Ass_{S}(S/I) &, \mbox{ if }d+2>n-d-1 \\
	\Ass_S(S/I)\cup\{(x_1,\ldots,x_n)\} &, \mbox{ if }d+2\leq n-d-1.\\
	\end{array}\right.
	\]
In particular, if $d+2>n-d-1$ then $I$ is normally torsion--free. Otherwise, if $d+2\leq n-d-1$, then $I^2$ is a normally torsion--free ideal.
\end{Theorem}

\begin{proof}
Let $k>1$ be an integer and assume that $d+2>n-d-1$. In this case, we prove that the graph $\bar G$ is bipartite, which is equivalent, by \cite{SVV}, with $\Ass_S(S/I^k)=\Ass_S(S/I)$.

Let $V_1=\{1,\ldots,n-d-1\}$ and $V_2=\{n-d,\ldots,n\}$, $V_1\cap V_2=\emptyset$. Let $\{i,j\}$ be an edge of $\bar G$, that is $j-i>d$. Since $i\geq 1$, we get that $j>d+i\geq d+1$. This implies that $j\geq d+2\geq n-d$, that is $j\in V_2$. Moreover, $i<j-d\leq n-d$ implies that $i\in V_1$. Therefore any edge of $\bar G$ has a vertex in $V_1$ and the other in $V_2$. Hence $\bar G$ is bipartite. In particular, it follows that $I$ is normally torsion--free.

Next, we assume that $d+2\leq n-d-1$ and we prove that $$\Ass_S(S/I^k)=\Ass_S(S/I)\cup\{(x_1,\ldots,x_n)\}.$$ 

For the inclusion "$\supseteq$", one has $\Ass_S(S/I^k)\supseteq \Ass_S(S/I)$, by \cite{BMV}. It remains to prove that $\frak m=(x_1,\ldots,x_n)\in \Ass_S(S/I^k)$, that is there is a monomial $m\in S/I^k$ such that $\frak m=I^k: m$. We analyze two cases:

\textit{Case 1.} If $k\leq d+2$, then, using the assumption $d+2\leq n-d-1$, we obtain $d+k<n$. We consider the monomial $m=x_1^{k-1}x_{d+2}\cdots x_{d+k}x_n$. We have that $\deg(m)=2k-1$ hence $m\notin G(I^k)$. For all $1\leq i\leq n$, we get $x_im\in I^k$. Indeed, if $i\leq d+1$, then $x_im=x_1^{k-1}x_ix_{d+2}\cdots x_{d+k}x_n\in G(I^k)$ since $n-i\geq n-d-1\geq d+2>d$. If $i=d+s$, for some $2\leq s\leq k$, then $x_im=x_1^{k-1}x_{d+2}\cdots x_i\cdots x_{d+k}x_n\in G(I^k)$ since $i-1=d+s-1>d$. Finally, if $d+k<i\leq n$, then $x_im=x_1^{k-1}x_{d+2}\cdots x_{d+k}x_ix_n\in G(I^k)$ since $i-1>d+k-1>d$ and $n-(d+2)\geq d+1>d$.

\textit{Case 2.} For $d+2<k$, we take $m=x_1\cdots x_{d+2}\cdots x_kx_{k+1}\cdots x_{2k-1}$. We observe that $m\notin G(I^k)$ since $\deg(m)=2k-1$. Then for all $1\leq i\leq n$ we obtain $x_im\in I^k$. Indeed, the assertion is clear for $i\leq 2k-1$. For $i>2k-1$, the monomial $x_im=x_1\cdots x_{d+2}\cdots x_kx_{k+1}\cdots x_{2k-1}x_i\in G(I^k)$ since $i-k>k-1>d$.

Therefore $\frak m=(x_1,\ldots,x_n)\in \Ass_S(S/I^k)$ and we get the desired inclusion.

Conversely, we have to prove that $\Ass_S(S/I^k)\subseteq \Ass_S(S/I)\cup\{(x_1,\ldots,x_n)\}$. Let $\frak p\in \Ass_S(S/I^k)$, that is $\frak p=I^k:m$, for some monomial $m\notin I^k$. We assume that $\frak p\subsetneq\frak m=(x_1,\ldots,x_n)$, thus there exists $x_i\notin \frak p$ and $i$ is minimal with this property.

We note that we must have $i\leq n-d$. Indeed, assume that $i>n-d$, hence $\frak p\supseteq(x_1,\ldots,x_{n-d})$. Then $x_{n-d}m\in I^k$, that is $x_{n-d}m=m_1\cdots m_kw$, with $m_1,\ldots,m_k\in G(I)$ and $w\in S$. Moreover, we have that $x_{n-d}\mid m_t$, for some $1\leq t\leq k$. Since every minimal monomial generator $u\in G(I)$ has the property that $\min(u)<n-d$, it results that $m_t=x_jx_{n-d}$, for some integer $j$ such that $n-d-j>d$. Then $x_im=m_1\cdots m_{t-1}m_{t+1}\cdots m_k(x_jx_i)w\in I^k$, since $x_jx_i\in G(I)$ having $i>n-d>j+d$. This implies that $x_i\in I^k:m=\frak p$, a contradiction, thus $i\leq n-d$.

Next, we prove that $\frak p=P_{[n]\setminus\{i,i+1,\ldots,i+d\}}$, hence $\frak p\in \Ass_S(S/I)$.

Let $x_j\in P_{[n]\setminus\{i,i+1,\ldots,i+d\}}$. By the minimality of $x_i$ we obtain $x_j\in \frak p$, if $j<i$. Otherwise, if $j>i+d$, we get $x_ix_j\in G(I)$ and $(x_ix_j)^k\in I^k$. Since $\frak p=I^k:m\supseteq I^k$, it results that $(x_ix_j)^k\in \frak p$. Therefore $x_j\in \frak p$, because $x_i\notin\frak p$. We proved that $\frak p\supseteq P_{[n]\setminus\{i,i+1,\ldots,i+d\}}$. 

It remains to prove that we cannot have $\frak p\supsetneq P_{[n]\setminus\{i,i+1,\ldots,i+d\}}$. In order to prove this, we need some more considerations.

One may note that $x_{i-1}\in \frak p$, by the minimality of $i$. Then $x_{i-1}m=u_1\cdots u_kw'$, with $u_1,\ldots,u_k\in G(I)$. Since $m\notin I^k$, we may assume, possibly after a renumbering, that $x_{i-1}\mid u_k$. Then $u_k=x_{i-1}x_l$, for some $l$ such that $l-(i-1)>d$, or $u_k=x_lx_{i-1}$, with $i-1-l>d$. Assume that we are in the second case, that is $u_k=x_lx_{i-1}$, with $i-1-l>d$. Then in particular $i-l>d$ and we obtain $x_im=u_1\cdots u_{k-1}(x_lx_i)w'\in I^k$, a contradiction with $x_i\notin \frak p$. Hence $u_k=x_{i-1}x_l$, with $l-(i-1)>d$. Moreover, if $l-i>d$, arguing as before, we obtain again a contradiction. Therefore we must have $u_k=x_{i-1}x_{i+d}$. Since $m=u_1\cdots u_{k-1}x_{i+d}w'$ and $m\notin I^k$, we get $\supp(w')\subseteq\{i,i+1,\ldots,i+d\}$. Indeed, if there exists an integer $s\in \supp(w')$ such that $s<i$, then $x_sx_{i+d}\in G(I)$, thus $m\in I^k$, and if $s>i+d$, then $m=u_1\cdots u_{k-1}(x_ix_s)w'/x_s\in I^k$. In both cases, we get a contradiction, thus $\supp(w')\subseteq\{i,i+1,\ldots,i+d\}$.

Let $1\leq s\leq k-1$ and $u_s=x_{a_s}x_{b_s}$, with $b_s-a_s>d$ such that $u_s\mid m$. We remark that if $a_s\leq i-1$ and $b_s>i+d$, then 
$$x_im=u_1\cdots u_{s-1}u_{s+1}\cdots u_{k-1}(x_ix_{b_s})(x_{a_s}x_{i+d})w'\in I^k,$$
a contradiction. Hence $a_s\geq i$ or $b_s\leq i+d$. This allow us to write
$$m=(x_{a_1}x_{b_1})\cdots (x_{a_s}x_{b_s})(x_{a_{s+1}}x_{b_{s+1}})\cdots (x_{a_{k-1}}x_{b_{k-1}})x_{i+d}w',$$ 
where $a_1,\ldots, a_s<i$ and $a_{s+1},\ldots, a_{k-1}\geq i$. Moreover, it results that $b_1,\ldots, b_s\leq i+d$. Using the fact that $b_j>a_j+d\geq i+d$, for all $s+1\leq j\leq k-1$, we get $b_{s+1},\ldots, b_{k-1}>i+d$. 

Firstly, in order to prove that $\{b_1,\ldots, b_s\}\subseteq\{i,\ldots, i+d\}$, assume by contradiction that $b_r<i$ for some $1\leq r\leq s$. This yields to $x_im\in I^k$, since
$$x_im=\left(\prod\limits_{1\leq j\neq r\leq s}(x_{a_j}x_{b_j})\right)(x_{a_r}x_i)(x_{b_r}x_{i+d})\left(\prod\limits_{s+1\leq j\leq k-1}(x_{a_j}x_{b_j})\right)w',$$ 
where $x_{a_r}x_i,x_{b_r}x_{i+d}\in G(I)$, a contradiction. Hence $b_1,\ldots, b_s\geq i$, thus $$\{b_1,\ldots, b_s\}\subseteq\{i,\ldots, i+d\}.$$

Secondly, we claim that $\{a_{s+1},\ldots,a_{k-1}\}\subseteq\{i,\ldots,i+d\}$. Assume by contradiction $a_r>i+d$ for some $s+1\leq r\leq k-1$. Then
$$x_im=\left(\prod\limits_{1\leq j\leq s}(x_{a_j}x_{b_j})\right)(x_ix_{a_r})(x_{i+d}x_{b_r})\left(\prod\limits_{s+1\leq j\neq r\leq k-1}(x_{a_j}x_{b_j})\right)w'\in I^k,$$ 
since $x_ix_{a_r},x_{i+d}x_{b_r}\in G(I)$, again a contradiction. Thus $$\{a_{s+1},\ldots,a_{k-1}\}\subseteq\{i,\ldots,i+d\}.$$

We conclude that $\frak p=I^k:m$, with 
$$m=(x_{a_1}x_{b_1})\cdots (x_{a_s}x_{b_s})(x_{a_{s+1}}x_{b_{s+1}})\cdots (x_{a_{k-1}}x_{b_{k-1}})x_{i+d}w',$$ 
$\supp(w')\subseteq\{i,\ldots,i+d\}$, $a_1,\ldots, a_s<i$, $b_{s+1},\ldots, b_{k-1}>i+d$, and $$\{b_1,\ldots, b_s,a_{s+1},\ldots,a_{k-1}\}\subseteq\{i,\ldots, i+d\}.$$ 

We claim that for all $j\in \{i,\ldots,i+d\}$, we get $x_jm\notin I^k$. This statement implies that $\frak p=P_{[n]\setminus\{i,i+1,\ldots,i+d\}}$.

Assume that $x_jm\in I^k$, for some $j\in \{i,\ldots,i+d\}$. Then
$$x_jm=(x_{a_1}x_{b_1})\cdots (x_{a_s}x_{b_s})(x_{a_{s+1}}x_{b_{s+1}})\cdots (x_{a_{k-1}}x_{b_{k-1}})x_{i+d}x_jw'\in I^k,$$ 
where $a_1,\ldots, a_s<i$, $\{j,b_1,\ldots, b_s,a_{s+1},\ldots,a_{k-1}\}\cup \supp(w')\subseteq\{i,\ldots, i+d\}$ and $b_{s+1},\ldots, b_{k-1}>i+d$. Then we can obtain at most $s$ minimal monomial generators of $I$, divisible by one of $a_1,\ldots, a_s$, and at most $k-s-1$ monomials belonging to $G(I)$, which are divisible by $b_{s+1},\ldots, b_{k-1}$. Thus $x_jm$ can be written as a product of at most $k-1$ minimal monomial generators of $I$, contradiction.

Hence $x_jm\notin I^k$, for all $j\in \{i,\ldots,i+d\}$, and we get $\frak p=P_{[n]\setminus\{i,i+1,\ldots,i+d\}}$, as desired.
\end{proof}


\begin{thebibliography}{99}

\bibitem {B} M. Brodmann, \textit{Asymptotic stability of $\Ass(M/I^nM)$}, Proc. Amer. Math. Soc., \textbf{74}(1979), 16–-18.

\bibitem{CH} A. Conca, J. Herzog, \textit{Castelnuovo--Mumford regularity of products of ideals}, Collect. Math., {\bf54}(2003), no. 2, 137--152.

\bibitem {Fe} D. Ferrarello, {\em The complement of a $d-$tree is Cohen--Macaulay}, Math. Scand., {\bf99}(2006), 161–-167. 

\bibitem {F} R. Fr\"oberg, \textit{On Stanley--Reisner rings}, in: Topics in algebra, Banach Center Publications, 26 Part 2, (1990), 57–-70.

\bibitem {HH} J. Herzog, T. Hibi, \textit{Monomial Ideals}, Graduate Texts in Mathematics, Springer--Verlag, 2011.

\bibitem {HHZ} J. Herzog, T. Hibi, X. Zheng, \textit{Monomial ideals whose powers have a linear resolution}, Math. Scand. \textbf{95}(2004), 23-–32.

\bibitem {HW} A.H. Hoefel, G. Whieldon, \textit{Linear quotients of square of the edge ideal of the anticycle}, arXiv: 1106.2348v2.

\bibitem {BMV} J. Martin\'ez--Bernal, S. Morey, R.H. Villarreal, \textit{Associated primes of powers of edge ideals}, arXiv: 1103.0992v3.

\bibitem{NP} E. Nevo, I. Peeva, \textit{Linear resolution of powers of edge ideals}, preprint 2009.

\bibitem {SVV} A. Simis, W.V. Vasconcelos, R.H. Villarreal, \textit{On the ideal theory of graphs}, J. Algebra, {\bf167}(1994), 389--416.

\bibitem{V} R.H. Villarreal, \textit{Cohen--Macaulay graphs}, Manuscripta Math., \textbf{66}(1990), no. 3, 277--293.

\end{thebibliography}
\end{document}